\documentclass[a4paper,fleqn,11pt]{amsart}

\usepackage{mathrsfs}
\usepackage{mathptmx}
\usepackage{cite}
\usepackage{amssymb}

\DeclareMathAlphabet{\mathpzc}{OT1}{pzc}{m}{it}

\newtheorem{theorem}{Theorem}[section]
\newtheorem{lemma}[theorem]{Lemma}

\newtheorem{corollary}[theorem]{Corollary}
\newtheorem{conjecture}[theorem]{Conjecture}

\theoremstyle{definition}
\newtheorem{definition}[theorem]{Definition}

\theoremstyle{remark}
\newtheorem{remark}[theorem]{Remark}

\numberwithin{equation}{section}

\allowdisplaybreaks[4]

\begin{document}

\title{Topologically equisingular deformations of homogeneous hypersurfaces with line singularities are equimultiple} 

\author{Christophe Eyral}
\address{C. Eyral, Institute of Mathematics, Polish Academy of Sciences, \'Sniadeckich~8, 00-656 Warsaw, Poland}
\email{eyralchr@yahoo.com}



\begin{abstract}
We prove that if $\{f_t\}$ is a family of line singularities with constant L\^e numbers and such that $f_0$ is a homogeneous polynomial, then $\{f_t\}$ is equimultiple. This extends to line singularities a well known theorem of A.~M. Gabri\`elov and A.~G. Ku\v{s}nirenko concerning isolated singularities. As an application, we show that if $\{f_t\}$ is a topologically $\mathscr{V}$-equisingular family of line singularities, with $f_0$ homogeneous, then $\{f_t\}$ is equimultiple. This provides a new partial positive answer to the famous Zariski multiplicity conjecture for a special class of non-isolated hypersurface singularities.
\end{abstract}

\maketitle

\markboth{C. Eyral}{Topologically equisingular deformations of homogeneous hypersurfaces with line singularities} 

\section{Introduction}\label{intro}

Let $\mathbf{z}:=(z_1,\ldots,z_n)$ be linear coordinates for $\mathbb{C}^n$ ($n\geq 2$), and let 
\begin{equation*}
f_0\colon (\mathbb{C}^n,\mathbf{0})\rightarrow (\mathbb{C},0),\quad 
\mathbf{z}\mapsto f_0(\mathbf{z}), 
\end{equation*}
be a \emph{homogeneous} polynomial function. We suppose that $f_0$ is reduced at $\mathbf{0}$. A deformation of $f_0$ is a polynomial function 
\begin{equation*}
f\colon (\mathbb{C} \times \mathbb{C}^n,\mathbb{C} \times \{\mathbf{0}\}) \rightarrow (\mathbb{C},0),\quad 
(t,\mathbf{z})\mapsto f(t,\mathbf{z}), 
\end{equation*}
such that the following two conditions hold:
\begin{enumerate} 
\item 
$f(0,\mathbf{z})=f_0(\mathbf{z})$ for any $\mathbf{z}\in\mathbb{C}^n$;
\item
if we write $f_t(\mathbf{z}):= f(t,\mathbf{z})$, then each $f_t$ is reduced at~$\mathbf{0}$.
\end{enumerate}
Thus a deformation of $f_0$ may be viewed as a $1$-parameter family of polynomial functions $f_t$ locally reduced at $\mathbf{0}$ and depending polynomially on the parameter~$t$. (Note that, by definition, $f_t(\mathbf{0}) = f(t,\mathbf{0})=0$ for any $t\in\mathbb{C}$.) In this paper, we~are interested in the embedded topology of the hypersurfaces $V(f_t):=f_t^{-1}(0)$ in a~neighbourhood of the origin $\mathbf{0}\in \mathbb{C}^n$ as the parameter $t$ varies from $t=t_0\not=0$ to $t=0$.
 
In \cite{GK}, A.~M. Gabri\`elov and A.~G. Ku\v{s}nirenko proved the following theorem.

\begin{theorem}[Gabri\`elov-Ku\v{s}nirenko]\label{gk}
Assume that $f_0$ is a homogeneous polynomial function with an isolated singularity at the origin. Also, suppose that the~family $\{f_t\}$ is a $\mu$-constant deformation of $f_0$---that is, for any parameter $t$ sufficiently small, the function $f_t$ has an isolated singularity at $\mathbf{0}$ and the Milnor number of $f_t$ at~$\mathbf{0}$ is independent of $t$. Under these conditions, the family $\{f_t\}$ is equimultiple---that is, the multiplicity of $f_t$ at $\mathbf{0}$ is independent of $t$ for all sufficiently small $t$.
\end{theorem} 

Here, by the \emph{multiplicity} of $f_t$ at $\mathbf{0}$ (denoted by $\mbox{mult}_{\mathbf{0}}(f_t)$) we mean the number of points of intersection near~$\mathbf{0}$ of~$V(f_t)$ with a generic line in~$\mathbb{C}^n$ passing arbitrarily close to, but not through, the origin. As $f_t$ is reduced at $\mathbf{0}$, this is also the order of~$f_t$ at~$\mathbf{0}$---that is, the lowest degree in the power series expansion of $f_t$ at $\mathbf{0}$. 

A key-point in the proof of Theorem \ref{gk} is the following theorem of L\^e D\~ung Tr\'ang and K. Saito (cf.~\cite{LS}).

\begin{theorem}[L\^e-Saito]\label{ls}
If the family $\{f_t\}$ is a $\mu$-constant family of isolated hypersurface singularities, then the $t$-axis $\mathbb{C}\times \{\mathbf{0}\}$ satisfies Thom's $a_f$ condition at the origin with respect to the ambient stratum---that is, if $\{\mathbf{p}_k\}$ is a sequence of points in $(\mathbb{C}\times \mathbb{C}^n)\setminus \Sigma f$ such that 
\begin{displaymath}
\mathbf{p}_k\rightarrow (0,\mathbf{0})
\quad\mbox{and}\quad
T_{\mathbf{p}_k}V(f-f(\mathbf{p}_k))\rightarrow T, 
\end{displaymath}
then $T_{(0,\mathbf{0})}(\mathbb{C}\times\{\mathbf{0}\}) = \mathbb{C}\times\{\mathbf{0}\} \subseteq T$. 
\end{theorem} 

In this theorem, $f_0$ is not required to be homogeneous.
As usual, $T_{\mathbf{p}_k} V(f-f(\mathbf{p}_k))$ denotes the tangent space at $\mathbf{p}_k$ to the level hypersurface in $\mathbb{C}\times \mathbb{C}^{n}$ defined by $f(t,\mathbf{z})=f(\mathbf{p}_k)$, and $T_{(0,\mathbf{0})}(\mathbb{C}\times\{\mathbf{0}\})$ is the tangent space at $(0,\mathbf{0})$ to $\mathbb{C}\times\{\mathbf{0}\}$. The notation $\Sigma f$ stands for the critical set of $f$.

Theorem \ref{gk} partially answers the following conjecture of 
B. Teissier \cite{T}. 

\begin{conjecture}[Teissier]
Any $\mu$-constant family of isolated hypersurface singularities is equimultiple.
\end{conjecture}

Note that by the L\^e-Ramanujam theorem \cite{LR}, Teissier's conjecture is a special case---at least when $n\not=3$---of the famous Zariski multiplicity conjecture \cite{Z}.

\begin{conjecture}[Zariski's multiplicity conjecture]\label{zmc}
Any topologically $\mathscr{V}$-equisin\-gular family of (possibly non-isolated) hypersurface singularities is equimultiple.
\end{conjecture}

Here, a family $\{f_t\}$ is said to be \emph{topologically $\mathscr{V}$-equisingular} if there exist open neighbourhoods $D$ and $U$ of the origins in $\mathbb{C}$ and $\mathbb{C}^n$, respectively, together with a continuous map $\varphi\colon (D\times U, D\times \{\mathbf{0}\})\rightarrow (\mathbb{C}^n,\mathbf{0})$ such that for all sufficiently small $t$, there is an open neighbourhood $U_t\subseteq U$ of $\mathbf{0}\in\mathbb{C}^n$ such that
the map $\varphi_t\colon (U_t,\mathbf{0})\rightarrow (\varphi(\{t\}\times U_t),\mathbf{0})$ defined by $\varphi_t(\mathbf{z}):=\varphi(t,\mathbf{z})$ is a homeomorphism sending $V(f_0)\cap U_t$ onto $V(f_t)\cap \varphi_t(U_t)$.

For a survey---up to 2007---on the Zariski multiplicity conjecture, we refer the reader to \cite{E2}. For more recent results and for a short introduction to equisingularity theory for non-isolated singularities, see \cite{EyBook}.

In the present paper, we investigate the same question as Gabri\`elov and Ku\v{s}nirenko for the simplest class of hypersurfaces with non-isolated singularities---namely, the hypersurfaces with \emph{line} singularities. Certainly, for such singularities, the Milnor number is no longer relevant. However, in \cite{M3,M4,M,M2}, D.~Massey introduced a series of polar invariants which generalizes to these singularities---actually to singularities of arbitrary dimension---the data contained by the Milnor number for an isolated singularity. These polar invariants are called the \emph{L\^e numbers}. Unlike the Milnor number, the L\^e numbers are not topological invariants (cf.~\cite{BG}). However, they still capture an important information about the singularity (see, e.g., \cite{Bo1,Bo2,BG,ER1,ER2,M7,M6,M3,M4,M,M2}). In the present work, we highlight new aspects of deformations with constant L\^e numbers---namely, we show that the theorem of Gabri\`elov and Ku\v{s}nirenko extends to line singularities provided that the constancy of the Milnor number is replaced by the constancy of the L\^e numbers. 

From now on, we suppose that $\{f_t\}$ is a family of line singularities. As in \cite[\S4]{M7}, by this we mean that for each $t$ near $0\in\mathbb{C}$ the singular locus $\Sigma f_t$ of $f_t$ near the origin $\mathbf{0}\in\mathbb{C}^n$ is given by the $z_1$-axis, and the restriction of $f_t$ to the hyperplane $V(z_1)$ defined by $z_1=0$ has an isolated singularity at the origin. Then, by \cite[Remark 1.29]{M}, the partition of $V(f_t)$ given by
\begin{equation*}
\mathscr{S}_t:=\bigl\{V(f_t)\setminus\Sigma f_t,\Sigma f_t\setminus\{\mathbf{0}\}, \{\mathbf{0}\}\bigr\}
\end{equation*}
is a \emph{good stratification} for $f_t$ in a neighbourhood of $\mathbf{0}$, and the hyperplane $V(z_1)$ is a \emph{prepolar} slice for $f_t$ at $\mathbf{0}$ with respect to $\mathscr{S}_t$ for all small $t$. In particular, combined with \cite[Proposition~1.23]{M}, this implies that the L\^e numbers $\lambda^0_{f_t,\mathbf{z}}(\mathbf{0})$ and $\lambda^1_{f_t,\mathbf{z}}(\mathbf{0})$ of $f_t$ at~$\mathbf{0}$ with respect to the coordinates~$\mathbf{z}$ do exist. (For the definitions of the L\^e numbers, good stratifications and prepolarity, we refer the reader to \cite[Chapter~1]{M}.)
Note that for line singularities, the only possible non-zero L\^e numbers are precisely $\lambda^0_{f_t,\mathbf{z}}(\mathbf{0})$ and $\lambda^1_{f_t,\mathbf{z}}(\mathbf{0})$; all the other L\^e numbers $\lambda^k_{f_t,\mathbf{z}}(\mathbf{0})$ for $2\leq k\leq n-1$ are defined and equal to zero (cf.~\cite{M}).

\begin{definition}
We say that the family $\{f_t\}$ is \emph{$\lambda_{\mathbf{z}}$-constant}---or $\lambda$-constant with respect to the coordinates $\mathbf{z}$---if for all sufficiently small $t$, the L\^e numbers
\begin{equation*}
\lambda^0_{f_t,\mathbf{z}}(\mathbf{0})
\quad\mbox{and}\quad
\lambda^1_{f_t,\mathbf{z}}(\mathbf{0})
\end{equation*}
of $f_t$ at~$\mathbf{0}$ with respect to~$\mathbf{z}$ are independent of $t$. 
\end{definition}

Requiring ``$\lambda_{\mathbf{z}}$-constant'' in a family of line singularities is a generalization of assuming ``$\mu$-constant'' in a family with isolated singularities (cf.~\cite{M}). 

Our generalization of the Gabri\`elov-Ku\v{s}nirenko theorem is the following.

\begin{theorem}\label{mt}
Suppose that $\{f_t\}$ is a family of line singularities. Also, suppose that the polynomial function $f_0$ is homogeneous. Under these assumptions, if, furthermore, the family $\{f_t\}$ is $\lambda_{\mathbf{z}}$-constant, then it is equimultiple.
\end{theorem} 

\begin{remark}
Assuming that the family $\{f_t\}$ is $\lambda_{\mathbf{z}}$-constant is not too strong in the sense that it does not imply the Whitney conditions along the $t$-axis (cf.~\cite{M7,M4}). (If these conditions were satisfied, then the theorem would immediately follow from a theorem of H. Hironaka \cite{H} which says that any reduced complex analytic space endowed with a Whitney stratification  is equimultiple along every stratum.)
\end{remark}

Theorem \ref{mt} is proved in Section \ref{sect-pmt}. 
A key-point in the proof is the following theorem of Massey (see \cite[Theorem 4.5]{M6} and \cite[Theorem 6.5]{M}) which extends to non-isolated singularities the L\^e-Saito theorem mentioned above.

\begin{theorem}[Massey]\label{Massey}
If $\{f_t\}$ is a $\lambda_{\mathbf{z}}$-constant family of (possibly non-isolated) hypersurface singularities, then the $t$-axis satisfies Thom's $a_f$ condition at the origin with respect to the ambient stratum.
\end{theorem}

In this theorem, the singular locus $\Sigma f_t$ of $f_t$ is not required to be $1$-dimensional at $\mathbf{0}$. If $s:=\dim_{\mathbf{0}} \Sigma f_0$, then we say that the family $\{f_t\}$ is $\lambda_{\mathbf{z}}$-constant if for all sufficiently small $t$, the L\^e numbers
\begin{equation*}
\lambda^0_{f_t,\mathbf{z}}(\mathbf{0}),
\ldots,
\lambda^{s}_{f_t,\mathbf{z}}(\mathbf{0})
\end{equation*}
of $f_t$ at~$\mathbf{0}$ with respect to~$\mathbf{z}$ are defined and independent of $t$. (Again, for $s+1\leq k\leq n-1$, $\lambda^k_{f_t,\mathbf{z}}(\mathbf{0})=0$, see\cite{M}.)

Theorem \ref{mt} has the following important corollary which provides a new partial positive answer to the Zariski multiplicity conjecture for a special class of non-isolated singularities. 

\begin{corollary}\label{mt3}
Suppose that $\{f_t\}$ is a family of line singularities. Also, suppose that the polynomial function $f_0$ is homogeneous. Under these assumptions, if, furthermore, the family $\{f_t\}$ is topologically $\mathscr{V}$-equisingular, then it is equimultiple.
\end{corollary}

\begin{proof}
From \cite[\S 4]{M7}, we know that the family $\{f_t\}$ is $\lambda_{\mathbf{z}}$-constant if and only if the Milnor number $\overset{\circ}{\mu}_{f_t}$ of a generic hyperplane slice of $f_t$ at a point on $\Sigma f_t$ sufficiently close to the origin and the reduced Euler characteristic of the Milnor fibre of $f_t$ at~$\mathbf{0}$ are both independent of $t$ for all small $t$. We also know that for line singularities,  $\overset{\circ}{\mu}_{f_t}$ is an invariant of the ambient topological type of $V(f_t)$ at $\mathbf{0}$ (cf.~\cite[\S 1]{M7}).
As the reduced Euler characteristic is of course a topological invariant too, and since our family $\{f_t\}$ is topologically $\mathscr{V}$-equisingular, it follows that $\{f_t\}$ is $\lambda_{\mathbf{z}}$-constant. Now we apply Theorem \ref{mt}.
\end{proof}

\section{Proof of Theorem \ref{mt}}\label{sect-pmt}

It is along the same lines as the proof of Theorem \ref{gk} given by A. M. Gabri\`elov and A. G. Ku\v{s}nirenko in \cite{GK} with two differences related to the fact that the singular set $\Sigma f_t$ is no longer zero-dimensional at $\mathbf{0}$. The first difference is that, unlike the Milnor number, the L\^e numbers depend on the choice of the coordinates, and it is possible that for a prescribed coordinates system, the L\^e numbers with respect to this system do not even exist. Actually, this is the only reason for the restriction of our study to line singularities (cf.~Lemma \ref{clnde}). 
The second difference concerns ``upper-semicontinuity.'' While the Milnor number in a family is upper-semicontinuous, the L\^e numbers need not be individually upper-semi\-continuous. They are only \emph{lexigraphically} upper-semi\-continuous (cf.~\cite[Example 2.10 and Corollary~4.16]{M}). Though this makes the argument a bit more complicated than in \cite{GK}, we shall see that it is enough to conclude. Finally, as mentioned in the introduction, let us recall that a key point in the proof is the theorem of Massey (Theorem \ref{Massey}) which plays a similar role as that of the L\^e-Saito theorem (Theorem~\ref{ls}) used by Gabri\`elov and Ku\v{s}nirenko in the case of isolated singularities.

We argue by contradiction. Write
\begin{equation*}
f(t,\mathbf{z})=f_0(\mathbf{z})+\sum_{1\leq i\leq r} p_i(t) g_i(\mathbf{z}),
\end{equation*}
where $g_i$ is a homogeneous polynomial of the variables $\mathbf{z}:=(z_1,\ldots,z_n)$, $p_i$ is a polynomial of the variable $t$, and $r$ is an integer. Set $d:=\deg(f_0)$ and $d_i:=\deg(g_i)$.
If the family $\{f_t\}$ is not equimultiple, then the map
\begin{equation*}
(t,\mathbf{z})\mapsto \sum_{{1\leq i\leq r}\atop {d_i<d}} p_i(t) g_i(\mathbf{z}) 
\end{equation*}
is not identically zero in an arbitrarily small neighbourhood of the origin. 
Write the polynomial $p_i(t)$ as 
\begin{equation*}
p_i(t)=a_i t^{k_i} + b_i t^{k_i+1} + \mbox{higher-order terms},
\end{equation*} 
where $k_i>0$ and $a_i,b_i$ are constants, and set
\begin{equation*}
\nu:=\sup_{1\leq i\leq r} \frac{\kappa(d-d_i)} {k_i},
\end{equation*}
where $\kappa$ is the lowest common multiple of $k_1,\ldots, k_r$.
Clearly, $\nu$ is a positive integer. By reordering, we may assume that there exists an integer $m$, $1\leq m\leq r$, such that $\kappa(d-d_i)/k_i=\nu$ for $1\leq i\leq m$ and $\kappa(d-d_i)/k_i<\nu$ for $m+1\leq i\leq r$. Then, consider the polynomial function
\begin{equation*}
h(t,\mathbf{z}):=f_0(\mathbf{z}) + \sum_{1\leq i\leq m} a_i t^{d-d_i} g_i(\mathbf{z}).
\end{equation*}
As usual, we write $h_{t}(\mathbf{z}):=h(t,\mathbf{z})$. Note that $h_0=f_0$. 

\begin{lemma}\label{clnde}
For any sufficiently small $t$, the L\^e numbers of the function $h_t$ at the origin with respect to the prescribed coordinates system $\mathbf{z}$ do exist.
\end{lemma}

\begin{proof}
As $h_0=f_0$, the singular set $\Sigma h_0$ is given by the $z_1$-axis while the singular set $\Sigma ({h_0}_{\mid V(z_1)})$ reduces to the origin. This implies that for all $t$ sufficiently small, $\dim_{\mathbf{0}}\Sigma h_t\leq 1$ and $\dim_{\mathbf{0}}\Sigma ({h_t}_{\mid V(z_1)})\leq 0$. Now, it follows from \cite[Remark 1.29 and Proposition 1.23]{M} that the L\^e numbers $\lambda^0_{h_t,\mathbf{z}}(\mathbf{0})$ and $\lambda^1_{h_t,\mathbf{z}}(\mathbf{0})$ are defined. 
\end{proof}

\begin{lemma}\label{gtnlc}
The family $\{h_t\}$ is not $\lambda_{\mathbf{z}}$-constant. 
\end{lemma}

\begin{proof}
Otherwise, by Theorem \ref{Massey}, for any holomorphic curve 
\begin{align*}
\gamma\colon (\mathbb{C},0) &\to (\mathbb{C} \times \mathbb{C}^n, (0,\mathbf{0}))\\
u&\mapsto\gamma(u)
\end{align*}
not contained in $\Sigma h$ in an arbitrarily small neighbourhood of the origin, we have
\begin{equation}\label{ordre}
\mbox{ord}_0 \biggl(\frac{\partial h}{\partial t} \circ\gamma \biggr) >
\inf_{1\leq j\leq n} \mbox{ord}_0 \biggl(\frac{\partial h}{\partial z_j} \circ\gamma \biggr),
\end{equation}
where $\mbox{ord}_0(\cdot)$ is the order at $0$ with respect to $u$.
However, if we take $\gamma(u):=(u,u\mathbf{z}_0)$, 
where $\mathbf{z}_0$ is such that 
\begin{align*}
\sum_{1\leq i\leq m} a_i (d-d_i) g_i(\mathbf{z}_0) \not=0,
\end{align*} 
then
\begin{align*}
\frac{\partial h}{\partial t} \circ \gamma(u) 
= \sum_{1\leq i\leq m} a_i (d-d_i) u^{d-d_i-1} g_i(u\mathbf{z}_0) 
= u^{d-1} \sum_{1\leq i\leq m} a_i (d-d_i)  g_i(\mathbf{z}_0),
\end{align*}
while, for all $1\leq j\leq n$,
\begin{align*}
\frac{\partial h}{\partial z_j} \circ \gamma(u) 
& = \frac{\partial f_0}{\partial z_j}(u\mathbf{z}_0) + 
\sum_{1\leq i\leq m} a_i u^{d-d_i}  
\frac{\partial g_i}{\partial z_j}(u\mathbf{z}_0) \\
& = u^{d-1} \biggl( \frac{\partial f_0}{\partial z_j}(\mathbf{z}_0) + 
\sum_{1\leq i\leq m} a_i \frac{\partial g_i}{\partial z_j}(\mathbf{z}_0) \biggr).
\end{align*}
Therefore, for this particular curve, the inequality (\ref{ordre}) is not satisfied.
\end{proof}

Lemmas \ref{clnde} and \ref{gtnlc} show that there exists a sequence $\{t_q\}$ of points in $\mathbb{C}\setminus \{0\}$ approaching the origin such that for all sufficiently large $q$ either 
$\lambda^{0}_{h_0,\mathbf{z}}(\mathbf{0}) \not=
\lambda^{0}_{h_{t_q},\mathbf{z}}(\mathbf{0})$
or $\lambda^{1}_{h_0,\mathbf{z}}(\mathbf{0}) \not=
\lambda^{1}_{h_{t_q},\mathbf{z}}(\mathbf{0}).$
Combined with the lexigraphically upper-semicontinuity (in the variable $t$) of the pair of L\^e numbers 
$(\lambda^{1}_{h_{t},\mathbf{z}}(\mathbf{0}), \lambda^{0}_{h_{t},\mathbf{z}}(\mathbf{0}))$, it follows that for any sufficiently large $q$, either
\begin{equation}\label{ifh1}
\lambda^{1}_{h_0,\mathbf{z}}(\mathbf{0}) >
\lambda^{1}_{h_{t_q},\mathbf{z}}(\mathbf{0})
\end{equation}
or
\begin{equation}\label{ifh2}
\lambda^{1}_{h_0,\mathbf{z}}(\mathbf{0}) =
\lambda^{1}_{h_{t_q},\mathbf{z}}(\mathbf{0})
\mbox{ and }
\lambda^{0}_{h_0,\mathbf{z}}(\mathbf{0}) >
\lambda^{0}_{h_{t_q},\mathbf{z}}(\mathbf{0}).
\end{equation}

For each $q$, choose a root $\tau_q$ of the polynomial $t^\kappa-t_q$, and for any $s\in\mathbb{C}$, $\mathbf{z}\in\mathbb{C}^n$, consider the function
\begin{equation*}
\ell^q(s,\mathbf{z}):=
\left\{
\begin{aligned}
& s^{-\kappa d} f((s\tau_q)^\nu, s^\kappa \mathbf{z}) &\mbox{for} && s\not=0,\\
& h_{t_q}(\mathbf{z}) &\mbox{for} && s=0.
\end{aligned}
\right.
\end{equation*}
It is easy to see that $\ell^q(s,\mathbf{z})$ is a polynomial in the variables $s,\mathbf{z}$. Indeed, for $m+1\leq i\leq r$, the integer $n_i:=\nu k_i-\kappa(d-d_i)$ is positive. Now, if we set $n_i:=0$ for $1\leq i\leq m$, then $\ell^q(s,\mathbf{z})$ is the polynomial given by
\begin{align*}
f_0(\mathbf{z}) + \sum_{1\leq i\leq r} (a_i s^{n_i} \tau_q^{\nu k_i} + b_i s^{n_i+\nu} \tau_q^{\nu(k_i+1)} + \mbox{higher-order terms}) g_i(\mathbf{z}).
\end{align*}
Write $\ell^q_{s}(\mathbf{z}):=\ell^q(s,\mathbf{z})$, and consider the $1$-parameter deformation family $\{\ell^q_{s}\}$ with respect to the parameter $s$. As $\ell^q_{0}=h_{t_q}$, for $q$ large enough we have $\dim_\mathbf{0}\Sigma \ell^q_{0}\leq 1$ (see the proof of Lemma \ref{clnde}). Furthermore, as the L\^e numbers are unchanged if we multiply the coordinates $\mathbf{z}$ by a non-zero constant, for any sufficiently large $q$ and any sufficiently small $s\not=0$ we have:
\begin{equation}\label{eln}
\lambda^{k}_{\ell^q_{s},\mathbf{z}}(\mathbf{0}) =
\lambda^{k}_{f_{(s\tau_q)^\nu},\mathbf{z}}(\mathbf{0})
\quad\mbox{for}\quad k\in\{0,1\}.
\end{equation}
Combined with the following lemma, the relation (\ref{eln}) shows that the family $\{f_t\}$ is not $\lambda_{\mathbf{z}}$-constant---a contradiction. (We recall that $h_0=f_0$.)

\begin{lemma}\label{claim}
For any sufficiently large $q$ and any sufficiently small $s\not=0$, either 
$\lambda^{1}_{h_0,\mathbf{z}}(\mathbf{0}) > 
\lambda^{1}_{\ell^q_{s},\mathbf{z}}(\mathbf{0})$
or
$\lambda^{0}_{h_0,\mathbf{z}}(\mathbf{0}) > 
\lambda^{0}_{\ell^q_{s},\mathbf{z}}(\mathbf{0}).$
\end{lemma}

\begin{proof}
The lexigraphically upper-semicontinuity (in the variable~$s$) of the pair of L\^e numbers $(\lambda^{1}_{\ell^q_{s},\mathbf{z}}(\mathbf{0}), \lambda^{0}_{\ell^q_{s},\mathbf{z}}(\mathbf{0}))$ shows that for all large $q$ and all small $s\not=0$, either
\begin{align}\label{case-1}
\lambda^{1}_{\ell^q_{0},\mathbf{z}}(\mathbf{0}) >
\lambda^{1}_{\ell^q_{s},\mathbf{z}}(\mathbf{0})
\end{align}
or
\begin{align}\label{case-2}
\lambda^{1}_{\ell^q_{0},\mathbf{z}}(\mathbf{0}) =
\lambda^{1}_{\ell^q_{s},\mathbf{z}}(\mathbf{0}) \mbox{ and }
\lambda^{0}_{\ell^q_{0},\mathbf{z}}(\mathbf{0}) \geq
\lambda^{0}_{\ell^q_{s},\mathbf{z}}(\mathbf{0}).
\end{align}
Furthermore, we know that for all large $q$, either (\ref{ifh1}) or (\ref{ifh2}) holds. If (\ref{ifh1}) holds, then, as either (\ref{case-1}) or (\ref{case-2}) holds too, we have:
\begin{equation*}
\lambda^{1}_{h_0,\mathbf{z}}(\mathbf{0}) >
\lambda^{1}_{h_{t_q},\mathbf{z}}(\mathbf{0}) =
\lambda^{1}_{\ell^q_{0},\mathbf{z}}(\mathbf{0}) \geq
\lambda^{1}_{\ell^q_{s},\mathbf{z}}(\mathbf{0}).
\end{equation*}
If (\ref{ifh2}) holds, then either
\begin{equation*}
\lambda^{1}_{h_0,\mathbf{z}}(\mathbf{0}) =
\lambda^{1}_{h_{t_q},\mathbf{z}}(\mathbf{0}) =
\lambda^{1}_{\ell^q_{0},\mathbf{z}}(\mathbf{0}) >
\lambda^{1}_{\ell^q_{s},\mathbf{z}}(\mathbf{0})
\end{equation*}
(when (\ref{case-1}) holds), or
\begin{equation*}
\lambda^{0}_{h_0,\mathbf{z}}(\mathbf{0}) >
\lambda^{0}_{h_{t_q},\mathbf{z}}(\mathbf{0}) =
\lambda^{0}_{\ell^q_{0},\mathbf{z}}(\mathbf{0}) \geq 
\lambda^{0}_{\ell^q_{s},\mathbf{z}}(\mathbf{0})
\end{equation*}
(when (\ref{case-2}) holds).
\end{proof}

\bibliographystyle{amsplain}

\end{document}